\documentclass[a4paper, 11pt]{article}
\usepackage{amssymb}

\usepackage{amsmath,amssymb,amsthm}
\usepackage[latin1]{inputenc}
\usepackage{version,tabularx,multicol}
\usepackage{graphicx,float,psfrag}
\usepackage{stmaryrd}
 \usepackage{color}
\usepackage[pdftex]{hyperref}
\usepackage{amsfonts}
\usepackage{mathrsfs}
\usepackage{geometry}
\usepackage{cite}
\usepackage[numbers,sort&compress]{natbib}

\allowdisplaybreaks[4]

\theoremstyle{plain}
\newtheorem{theorem}{Theorem}[section]

 \newtheorem{proposition}[theorem]{Proposition}

 \newtheorem{lemma}[theorem]{Lemma}

\newtheorem{remark}{Remark}[section]

\newcommand{\ind}{{\bf 1}}

  \makeatletter
  \@addtoreset{equation}{section}
  \makeatother

 \def\beqlb{\begin{eqnarray}}\def\eeqlb{\end{eqnarray}}
 \def\beqnn{\begin{eqnarray*}}\def\eeqnn{\end{eqnarray*}}

\newcommand{\bcen}{\begin{center}}
\newcommand{\ecen}{\end{center}}
\newcommand{\bgeqn}{\begin{equation}}
\newcommand{\edeqn}{\end{equation}}

\begin{document}
\title{On the empty balls of a critical or subcritical branching random walk}
\author{Jie Xiong, Shuxiong Zhang}

\maketitle

\noindent\textit{Abstract:} Let $\{Z_n\}_{n\geq 0 }$ be a critical or subcritical $d$-dimensional branching random walk started from a Poisson random measure whose intensity  measure is the Lebesugue measure on $\mathbb{R}^d$. Denote by $R_n:=\sup\{u>0:Z_n(\{x\in\mathbb{R}^d:|x|<u\})=0\}$ the radius of the largest empty ball centered at the origin of $Z_n$. In this work, we prove that after suitable renormalization, $R_n$ converges in law to some non-degenerate distribution as $n\to\infty$. Furthermore, our work shows that the renormalization scales depend on the offspring law and the dimension of the branching random walk, which  completes the results of \cite{reves02} for the critical binary branching Wiener process.


\section{Introduction and Main results}
\subsection{Introduction}
 In this work, we consider a branching random walk (BRW) model $\{Z_n\}_{n\geq0}$ started from the Poisson random measure with Lebesgue intensity. This model is a measure valued process, which is governed by a probability distribution $\{p_k\}_{k\geq 0}$ on natural numbers (called the offspring distribution) and a $\mathbb{R}^d$-valued random vector $X$ (called the step size or displacement). Let us define it in the following way.
 \par
 At time $0$, there exist infinite many particles distributed according to the Poisson random measure, i.e. for any Borel measurable set $A\subset\mathbb{R}^d$,
$$\mathbb{P}(Z_0(A)=k)=\frac{|A|^k}{k!}e^{-|A|},~k\geq0,$$
 where $|\cdot|$ stands for the Lebesgue measure, and by convention $0=|A|^ke^{-|A|}$ if $|A|=\infty$. Then, these  particles die and produce offsprings independently according to the offspring distribution $\{p_k\}_{k\geq 0}$ in a instant. Afterwards, the offspring particles move independently according to the law of $X$ in unit time. This forms a random measure at time $1$, denoted by $Z_1$. Write $u\in Z_i$ if $u$ is a particle at time $i$ in the branching random walk $\{Z_n\}_{n\geq0}$. Denote by $S_u$ the position of particle $u$. Let $N_v,~v\in Z_n,~n\geq0$ be independent random variables with common distribution $\{p_k\}_{k\geq 0}$ and $X^v_i, 1\leq i\leq N_v,~v\in Z_n,~n\geq0$ be independent copies of $X$. In words, $N_v$ stands for the children number of $v$ and $X^v_i$ stands for the displacement of the $i$-th child of $v$. Hence,
 $$Z_1=\sum_{u\in Z_0}\sum^{N_u}_{i=1}\delta_{S_u+X^u_i}.$$
Generally, the measure valued process $Z_n$, $n\geq 2$ is defined  by the following iteration
$$Z_n:=\sum_{x\in Z_{n-1}}\tilde{Z}_1^{x},$$
where
$$\tilde{Z}_1^{x}:=\sum^{N_x}_{i=1}\delta_{S_x+X^x_i},~x\in Z_{n-1}$$
(conditioned on $Z_{n-1}$) are independent.
\par

Let $m:=\sum_{k\geq0}kp_k$ be the mean of the offspring law. We call $\{Z_n\}_{n\geq0}$
 a supercritical (critical, subcritical) branching random walk if $m>1$ ($= 1$, $< 1$). In the remainder of
this paper, we always consider the critical and subcritical case (i.e. $m\leq1$). In order to avoid trivialities, we always assume that
 $$p_0<1,~p_1<1.$$
 Denote by $|Z_n|:=Z_n(\mathbb{R})$ the total population at generation $n$. Let $\mathbb{P}_{\mu}$ be the probability measure under which a measure valued process (say, branching random walk or super-Brownian motion; see Section \ref{sectionthdim2} below) starts from the measure $\mu$. We write $\mathbb{P}=\mathbb{P}_{\mu}$ if $\mu$ is the Poisson random measure with Lebesgue intensity.
  \par
  Since $|Z_0|=\infty$, $\mathbb{P}$-almost surely, for any integer $k,n\geq 1$,
$$\mathbb{P}(|Z_n|>0)\geq\mathbb{P}_{k\delta_0}(|Z_n|>0)=1-(\mathbb{P}_{\delta_0}(|Z_n|=0))^k,$$
where the right hand side converges to $1$ as $k\to\infty$. Therefore, under $\mathbb{P}$, $\{Z_n\}_{n\geq0}$ survives almost surely. So, the following notation $R_n$ is well-defined. Let $B(u):=\{u:|x|<u\}$ be the $d-$dimensional open ball with radius $u$. Denote by $$R_n:=\sup\{u>0:Z_n(B(u))=0\}$$ the radius of the largest empty ball centered at the origin of $Z_n$, where we take $0=\sup\emptyset$ by convention. Or, in other words, $R_n$ is the shortest distance of the particles at time $n$ from the origin.
\par
In this paper, we aim at finding suitable renormalization scales $\{a_n\}_{n\geq0}$ such that for any $r>0$,
$$
\lim_{n\to\infty}\mathbb{P}\left(\frac{R_n}{a_n}\geq r\right)=e^{-F_d(r)},
$$
where $F_d(r)\in(0,\infty)$.
\par
The research on this question was first conducted by \cite{reves02} for the binary branching Wiener process (i.e. $p_0=p_1=1/2$ and $X$ is a standard normal random vector). He proved that if $d=1$, then $R_n/n$ converges in law as $n\to\infty$. For $d =2$ and $d\geq3$, he gave two
conjectures; see Remark \ref{4reer43} and Remark \ref{3ewe3rre} below. Later, \cite{hu05} partially confirmed R\'ev\'esz's conjecture for $d\geq3$; also see \cite{xz21} for this question about the  critical super-Brownian motion model.
\par
Our method to the proof is different from R\'ev\'esz. We first use the Laplace functional formula of the Poisson random measure to show that
$$\mathbb{P}(R_n\geq a_nr)=\exp\left\{-\int_{\mathbb{R}^d} \mathbb{P}_{\delta_{x}}(Z_n(B(a_nr))>0)dx\right\}.$$
Then, we prove the theorems by analyzing lower bounds and upper bounds of $\int_{\mathbb{R}^d} \mathbb{P}_{\delta_{x}}(Z_n(B(a_nr))>0)dx$. In case $d=1$, for the lower bound, we force every branching random walk started from $B(nr)$ to have one particle locate in $B(nr)$ at time $n$. For the upper bound, the key is to use the deviation probabilities of a random walk to argue that branching random walks emanating out from $B(n(r+\delta))$ can hardly reach $B(nr)$ at time $n$, where $\delta>0$ is a small constant. In case $d=2$, using the fact that the scaling limits of the branching random walk is a super-Brownian motion and results of empty balls for the super-Brownian motion (see \cite{xz21}), we obtain the desired lower bound. The upper bound is obtained in virtue of the extreme value theory for branching random walks (see \cite{lalley2015}) and the Poisson
cluster representation of super-Brownian motions. In case $d\geq3$, we use Paly-Zygmund and Markov inequalities to prove the lower bound and upper bound, respectively. We also consider the subcritical case. Unlike the critical case, $R_n$ increases exponentially in the subcritical case.

\subsection{Main results}
Write $\sigma^2:=\sum^\infty_{k=0}k^2p_k-\left[\sum^\infty_{k=0}kp_k\right]^2$. In the following of this paper, to simplify our statement, we always assume that the step size $X$ is a mean zero random vector and each component of $X$ are identically independent distributed. The first theorem consider the $1$-dimensional critical branching random walk.
\begin{theorem}\label{thdim1}Assume $m=1$, $\sigma^2<\infty$ and $d=1$. Suppose that $\mathbb{E}[|X|^{\alpha}]<\infty$ for some $\alpha>2$. Then for $r>0$,

$$\lim_{n\to\infty}\mathbb{P}\left(\frac{R_n}{n}\geq r\right)=e^{-\frac{4r}{\sigma^2}}.$$

\end{theorem}

The next theorem considers the case of $d=2$. Denote by $\{X_t\}_{t\geq0}$ a critical $d$-dimensional super-Brownian motion with branching mechanism $\psi(u)=\sigma^2u^2$ (for its serious definition see Section \ref{sectionthdim2} below). Define $B(x,r):=\{y\in\mathbb{R}^d:|y-x|<r\}$.
\begin{theorem}\label{thdim2}Assume $m=1$, $\sigma^2<\infty$ and $d=2$. Suppose $\mathbb{E}[|X|^4]<\infty$ and the correlation coefficient matrix of $X$ is the identity matrix. Then there exists a function $C(\cdot)>0$ such that $\lim_{\varepsilon\to0+}C(\varepsilon)=\infty$, and for any $r,\delta>0$,
\begin{align}
\exp\left\{ -\frac{2\pi(1+\delta)r^2}{\sigma^2}-C(\delta)  \right\}
&\leq \liminf_{n\to\infty}\mathbb{P}\left(\frac{R_n}{\sqrt n}\geq r\right)\cr
&\leq\limsup_{n\to\infty}\mathbb{P}\left(\frac{R_n}{\sqrt n}\geq r\right)\leq \exp\left\{\int_{\mathbb{R}^2}\log\mathbb{P}_{\delta_0}\left(X_1(B(x,r))=0\right)dx\right\}.\nonumber
\end{align}
\end{theorem}
\begin{remark}\label{4reer43}
From \cite[Theorem 1.2]{xz21}, we have
$$\lim_{r\to\infty}\frac{\int_{\mathbb{R}^2}-\log\mathbb{P}_{\delta_0}\left(X_1(B(x,r))=0\right)dx}{2\pi r^2/\sigma^2}=1.$$
In \cite[Conjecture 1]{reves02}, R\'ev\'esz conjectured that for a two-dimensional critical binary branching Wiener process (thus $\sigma^2=1$),
$$
\lim_{n\to\infty}\mathbb{P}\left(\frac{R_n}{\sqrt n}\geq r\right)=e^{-F_2(r)},
$$
where $F_2(r)\in(0,\infty)$ satisfying
\begin{align}\label{6ytgwevde}
\lim_{r\to\infty}\frac{F_2(r)}{2\pi r^2}=1.
\end{align}
Unfortunately, we can not prove $\lim_{n\to\infty}\mathbb{P}\left(\frac{R_n}{\sqrt n}\geq r\right)$ exists. But if proved, (\ref{6ytgwevde}) is a direct consequence of Theorem \ref{thdim2}. Thus, in a sense, our lower bound and upper bound are relatively sharp.
\end{remark}
\begin{remark}
Observe that under $\mathbb{P}_{\delta_0}\big( \cdot \big||Z_n|>0\big)$,
 $\left\{\frac{1}{n}Z_{\lfloor nt\rfloor}(\sqrt n\cdot)\right\}_{t\geq0}$ converges weakly to some measure-valued process $\{Y_t\}_{t\geq 0}$ as $n\to\infty$, where $\{Y_t\}_{t\geq 0}$ is related to $\{X_t\}_{t\geq 0}$ by the Poisson cluster representation; see (\ref{4f3w2e2323}) below. On the other hand, from (\ref{jitbndf12dim2}) below, we have
$$
\mathbb{P}\left(\frac{R_n}{\sqrt n}\geq r\right)=\exp\left\{-n\mathbb{P}_{\delta_0}(|Z_n|>0)\int_{\mathbb{R}^2}\mathbb{P}_{\delta_0}\left(\frac{1}{n}Z_n(\sqrt nB(x,r))>0\Big||Z_n|>0\right)dx \right\}.
$$
Thus, under the assumption of Theorem \ref{thdim2}, we conjecture that
 $$\lim_{n\to\infty}\mathbb{P}\left(\frac{R_n}{\sqrt n}\geq r\right)=\exp\left\{\int_{\mathbb{R}^2}\log\mathbb{P}_{\delta_0}\left(X_1(B(x,r))=0\right)dx\right\}$$
 (namely, the upper bound is sharp). However, since $\ind_{\{z>0\}}$ is not a continuous function on $[0,\infty)$, the conjecture is not a direct corollary of the weak convergence.
\end{remark}
 Denote by $v_d(r):=\frac{\pi^{d/2}r^d}{\Gamma(d/2+1)}$ (where $\Gamma(\cdot)$ is the so-called Gamma function) the volume of a $d$-dimensional ball with radius $r$. Let $X^{(1)}$ be the first component of vector $X$.
\begin{theorem}\label{thdim3}
Assume $m=1$, $\sigma^2<\infty$ and $d\geq3$. Suppose that $\mathbb{E}[|X|^3]<\infty$. Then there exists a positive function $F_d(r)$, $r>0$ such that

$$\lim_{n\to\infty}\mathbb{P}\left(R_n\geq r\right)=e^{-F_d(r)}.$$
Moreover,
$$\frac{v_d(r)}{1+\sigma^2 C_d(r)r^2}\leq F_d(r)\leq v_d(r),$$
where $C_d(r):=\frac{2\Big[1+6\mathbb{E}[|X^{(1)}|^3]/\left(r\mathbb{E}[|X^{(1)}|^2]^{3/2}\right)\Big]^d}{(d-2)}+1.$
\end{theorem}
\begin{remark} \label{3ewe3rre}
In Theorem \ref{thdim3}, one can see that there exists a constant $C(d,\sigma)>0$ such that for $r>1$,
$$C(d,\sigma)r^{d-2}\leq F_d(r)\leq v_d(1)r^d.$$
R\'ev\'esz \cite[Conjecture 1]{reves02} conjectured that for a $d$-dimensional ($d\geq3$) critical binary branching Wiener process,
$$\lim_{r\to\infty}\frac{F_d(r)}{C_dr^{d-2}}=1,$$
where $C_d>0$ is a constant satisfying $\lim_{d\to\infty}C_d/[v_d(1)^{(d-2)/d}]=1$. According to our results for the super-Brownian motion \cite[Theorem 3]{xz21}, we think this conjecture is true. Unfortunately, our method to deal with high-dimensional ($d\geq3$) branching random walks is not delicate enough, especially the lower bound of $\mathbb{P}\left(R_n\geq r\right)$.
\end{remark}
\begin{remark} Hu \cite{hu05} studied this question for a $d$-dimensional ($d\geq3$) critical binary branching Wiener process (i.e. $p_0=p_1=\frac{1}{2}$, $X$ is a normal random vector). Our result is similar to those of Hu, but the conditions are much weaker.
\end{remark}
The above three theorems consider the case of $\sigma^2<\infty$. A natural question is to consider the case of $\sigma^2=\infty$. For this purpose, we consider the case that the offspring law is in the domain attraction of an $(1+\beta)$-stable law, where $\beta\in(0,1)$. Unlike the case of $\sigma^2<\infty$, in this case, the renormalization scale depends on the offspring law parameter $\beta$. Denote by $f(s):=\sum^{\infty}_{k=0}p_ks^k,~s\in[0,1]$ the generating function of the offspring law, and let $L(s),s\in[0,1]$ be a slowly varying function as $s\to0+$ (i.e. $\lim_{s\to0+}L(sx)/L(s)=1$ for $x\in(0,1]$).
\begin{theorem}\label{thdimsgimawuqiong}Assume $m=1$ and $f(s)=s+(1-s)^{1+\beta}L(1-s)$ for some $\beta\in\left(0,\frac{1}{d}\right]$. Suppose that $\mathbb{E}[|X|^{\alpha}]<\infty$ for some $\alpha>(\beta+1)d$. Then for $r>0$,
$$
\lim_{n\to\infty}\mathbb{P}\left(\frac{R_n}{b_n}\geq r\right)=\exp\left\{-v_d(r)(\beta^{-1})^{\beta^{-1}}\right\},
$$
where
$$b_n:=\Big[nL\big(\mathbb{P}_{\delta_0}(|Z_n|>0)\big)\Big]^{\frac{1}{\beta d}},~n\geq1.$$
\end{theorem}
The above four theorems consider the critical case $m=1$. It turns out that $R_n$ basically grows like a power function. However, the next theorem tells us that, in the subcritical case (i.e. $m<1$), $R_n$ increases exponentially. The result holds in a very general setting.
\begin{theorem}\label{thdim-subcrit}Assume $m<1$, $\sum_{k\geq0}p_kk\log k <\infty$ and $d\geq1$. Suppose that $\mathbb{E}\left[|X|^{\alpha}\right]<\infty$ for some $\alpha>1$. Then for $r>0$,
$$\lim_{n\to\infty}\mathbb{P}\left(\frac{R_n}{(1/m)^\frac{n}{d}}\geq r\right)=e^{-Q(0)v_d(r)},$$
where $Q(s),~s\in[0,1]$ is the unique solution of the functional equation (see \cite[p40]{r1}):
\begin{align}
\begin{cases}
&Q(f(s))=mf(s);\cr
&Q(1)=0,\lim\limits_{s\to1}Q'(s)=1.\nonumber
\end{cases}
\end{align}
Moreover $Q(s)>0$ for $s\in[0,1)$.
\end{theorem}
The rest of this paper is organized as follows. Sections \ref{secthdim1}-\ref{secthdimddayu3} are devoted to study the case $m=1$ and $\sigma^2<\infty$, where Theorems \ref{thdim1}-\ref{thdim3} are proved, respectively. We shall consider the case $m=1$ and $\sigma^2=\infty$ in Section \ref{secsigmawuqiong}, and Theorem \ref{thdimsgimawuqiong} is proved in this section. We study the case $m<1$ in Section \ref{secsubcrtical}, where Theorem \ref{thdim-subcrit} is proved.
\section{Proof of Theorem \ref{thdim1}: $m=1$, $d=1$}\label{secthdim1}
Let $\{W_n\}_{n\geq0}$ be a random walk with increment $X$. For the remainder of the paper, we use $\mathbb{P}_x$ (and $\mathbb{E}_x$) to mean that the random walk $\{W_n\}_{n\geq0}$ (or Brownian motion) starts from $x$, and write $\mathbb{P}=\mathbb{P}_0$ for short. We first give a lemma which concerns the deviation probabilities of a random walk.
\begin{lemma}\label{Nagaevlem} Assume $d=1$. If $\mathbb{E}[|X|^{\alpha}]<\infty$ for some $\alpha\geq 1$, then there exists a constant $C_1>0$ such that for any $x>0$ and $n\geq1$,
$$\mathbb{P}(|W_n|\geq xn)\leq \frac{C_1}{n^{\alpha-1}x^\alpha}.$$
\end{lemma}
\begin{proof}
 For the case of $\alpha\geq 2$, the lemma is a direct consequence of \cite[Corollary 1.8]{Nagaev}. In the next, we only deal with the case $\alpha\in[1,2]$. Let $X_i,~i\geq 1$ be i.i.d. copies of $X$. From \cite[Corollary 1.6]{Nagaev}: if $A^+_{\alpha}:=\sum^n_{i=1}\mathbb{E}\left[X_i^{\alpha}\ind_{\{X_i\geq0\}}\right]<\infty$, then for any $z,~y>0$ satisfying $y^{\alpha}\geq 4A^{+}_{\alpha}$ and $z>y$,
$$
\mathbb{P}(W_n\geq z)\leq n\mathbb{P}(X>y)+\left(\frac{eA^+_{\alpha}}{zy^{\alpha-1}}\right)^{\frac{z}{2y}}.
$$
Let $z=nx$, $y=z/2$. Hence, if $(nx/2)^{\alpha}\geq 4n\mathbb{E}\left[|X|^{\alpha}\right]$, then
\begin{align}
\mathbb{P}(W_n\geq xn)&\leq n\mathbb{P}(X>nx/2)+\frac{eA^+_{\alpha}}{nx(nx/2)^{\alpha-1}}\cr
&\leq \frac{2^{\alpha}n\mathbb{E}\left[|X|^{\alpha}\right]}{(nx)^{\alpha}}+\frac{2^{\alpha-1}en\mathbb{E}\left[|X|^{\alpha}\right]}{(nx)^{\alpha}}\cr
&\leq \frac{2^{\alpha}\mathbb{E}\left[|X|^{\alpha}\right]+2^{\alpha-1}e\mathbb{E}\left[|X|^{\alpha}\right]}{n^{\alpha-1}x^\alpha}.
\end{align}
Thus, the lemma follows.
\end{proof}
 In literature, we call $\{|Z_n|\}_{n\geq 0}$, under $\mathbb{P}_{\delta_0}$, a branching process or Galton-Watson process; see \cite{r1}. It is well known that a branching process with a critical or subcritical offspring law will become extinct with probability $1$. Thus, the survival probability $\mathbb{P}_{\delta_0}(|Z_n|>0)$ converges to $0$ as $n\to\infty$. The following lemma concerns the decay rate of it. Recall that  $\sigma^2:=\sum^\infty_{k=0}k^2p_k-\left[\sum^\infty_{k=0}kp_k\right]^2$, $f(s)=\sum^\infty_{k=0}p_ks^k$ and $L(s)$ is a slowly varying function at $0$.
\begin{lemma}\label{extinpro1}(survival probabilities) \\
(i) If $m=1$ and $\sigma^2<\infty$, then
$$
\lim_{n\to\infty}n\mathbb{P}_{\delta_0}(|Z_n|>0)=\frac{2}{\sigma^2}.
$$
(ii) If $m=1$ and $f(s)=s+(1-s)^{1+\beta}L(1-s)$ for some $\beta\in(0,1]$, then
$$
\lim_{n\to\infty}n\mathbb{P}_{\delta_0}(|Z_n|>0)^{\beta}L(\mathbb{P}_{\delta_0}(|Z_n|>0))=\frac{1}{\beta}.
$$
(iii) If $m<1$ and $\sum_{k\geq0}k\log kp_k<\infty$, then
$$
\lim_{n\to\infty}\frac{1}{m^n}\mathbb{P}_{\delta_0}(|Z_n|>0)=Q(0)\in(0,\infty),
$$
where $Q(\cdot)$ is given in Theorem \ref{thdim-subcrit}.
\end{lemma}
\begin{proof}
(i) and (iii) come from \cite[p19]{r1} and \cite[p40]{r1}, respectively. (ii) follows from \cite[Lemma 2]{Slack}.
\end{proof}

Now we are ready to prove Theorem \ref{thdim1}.

\begin{proof}
\textbf{Upper bound.} Recall that for a particle $u$, $S_u$ stands for its position. Denote by $\{Z^{\delta_{x}}_n\}_{n\geq0}$ the branching random walk started from a single particle at position $x$. Recall that we write $u\in Z_0$ if $u$ is a particle at time $0$. Since $Z_0$ is a Poisson random measure, by the branching property,
\begin{align}\label{jitbndf12}
\mathbb{P}(R_n\geq nr)&=\mathbb{P}\left(Z_n(B(nr))=0\right)\cr
&=\mathbb{P}\left(\forall u\in Z_0,Z^{\delta_{S_u}}_n(B(nr))=0\right)\cr
&=\mathbb{E}\left[\Pi_{u\in Z_0}\mathbb{P}_{\delta_{S_u}}(Z_n(B(nr))=0)\right]\cr
&=\mathbb{E}\left[e^{\sum_{u\in Z_0}\log \mathbb{P}_{\delta_{S_u}}(Z_n(B(nr))=0)}\right]\cr
&=\mathbb{E}\left[e^{-\int_{\mathbb{R}}-\log \mathbb{P}_{\delta_{x}}(Z_n(B(nr))=0)Z_0(dx)}\right]\cr
&=e^{-\int_{\mathbb{R}} \mathbb{P}_{\delta_{x}}(Z_n(B(nr))>0)dx},
\end{align}
where the last equality follows from the Laplace functional formula of the Poisson random measure; see \cite[p19]{Bovier}.
 Let $L_n:=\min\{|S_u|:u\in Z_n\}$, where we define $+\infty=\min\emptyset$ by convention. Recall that $\{W_n\}_{n\geq0}$ stands for a random walk with increment $X$, and we use $\mathbb{P}_x$ to mean that $\{W_n\}_{n\geq0}$ starts from $x$ and write $\mathbb{P}=\mathbb{P}_0$ for short. Observe that
\begin{align}
\mathbb{P}_{\delta_x}(Z_n(B(nr))>0)&=\mathbb{P}_{\delta_x}\left(Z_n(B(nr))>0\big||Z_n|>0\right)\mathbb{P}_{\delta_x}(|Z_n|>0)\cr
&=\left[1-\mathbb{P}_{\delta_x}\left(Z_n(B(nr))=0\big||Z_n|>0\right)\right]\mathbb{P}_{\delta_x}(|Z_n|>0)\cr
&=\left[1-\mathbb{P}_{\delta_x}\left(L_n\geq nr\big||Z_n|>0\right)\right]\mathbb{P}_{\delta_0}(|Z_n|>0)\cr
&\geq\left[1-\mathbb{P}_{\delta_x}\left(|W_n|\geq nr\big||Z_n|>0\right)\right]\mathbb{P}_{\delta_0}(|Z_n|>0)\cr
&=\mathbb{P}_0\left(|x+W_n|< nr\right)\mathbb{P}_{\delta_0}(|Z_n|>0),\nonumber
\end{align}
where the inequality follows from the fact that $|W_n|\geq L_n$ and the last equality follows from the independent of the branching and the motion.
\par
 Therefore, for $\delta\in(0,r)$, we have
\begin{align}\label{tg3452}
\int_\mathbb{R}\mathbb{P}_{\delta_x}(Z_n(B(nr))>0)dx
&\geq\int_\mathbb{R}\mathbb{P}_{\delta_0}(|Z_n|>0)\mathbb{P}(|W_n+x|< nr)dx\cr
&\geq \int_{|x|\leq n(r-\delta)}\mathbb{P}_{\delta_0}(|Z_n|>0)\mathbb{P}(|W_n+x|< nr)dx\cr
&= \int_{|y|\leq r-\delta }n\mathbb{P}_{\delta_0}(|Z_n|>0)\mathbb{P}(|W_n+ny|< nr)dy\cr
&\geq \int_{|y|\leq r-\delta}n\mathbb{P}_{\delta_0}(|Z_n|>0)\mathbb{P}\left(\frac{|W_n|}{n}<\delta\right)dy\cr
&=2(r-\delta)n\mathbb{P}_{\delta_0}(|Z_n|>0)\mathbb{P}\left(\frac{|W_n|}{n}<\delta\right).
\end{align}
Hence, by the law of large numbers and Lemma \ref{extinpro1} (i), we have
$$\liminf_{n\to\infty}\int_{\mathbb{R}}\mathbb{P}_{\delta_x}(Z_n(B(nr))>0)dx\geq \frac{4(r-\delta)}{\sigma^2}.$$
\par
Plugging above into (\ref{jitbndf12}) and letting $\delta\to0$ give that
\begin{align}
\limsup_{n\to\infty}\mathbb{P}(R_n\geq nr)\leq e^{-\frac{4r}{\sigma^2}}.\nonumber
\end{align}

\smallskip

\noindent\textbf{Lower bound.} Let $\delta\in(0,\infty)$. Note that
\begin{align}\label{wicm45}
&\int_\mathbb{R}\mathbb{P}_{\delta_x}(Z_n(B(nr))>0)dx\cr
&\leq\int_{|x|\leq n(r+\delta)}\mathbb{P}_{\delta_0}(|Z_n|>0)dx+\int_{|x|>n(r+\delta)}\mathbb{P}_{\delta_x}(Z_n(B(nr))>0)dx\cr
&=2\mathbb{P}_{\delta_0}(|Z_n|>0)n(r+\delta)+\int_{|x|>n(r+\delta)}\mathbb{P}_{\delta_x}(Z_n(B(nr))>0)dx.
\end{align}
By the Markov inequality,
\begin{align}\label{7t6ij6y67}
\mathbb{P}_{\delta_x}(Z_n(B(nr))>0)&\leq\mathbb{E}_{\delta_x}[Z_n(B(nr))]\cr
&= \mathbb{E}_{\delta_x}[|Z_n|]\mathbb{P}_{x}(|W_n|\leq nr)\cr
&=\mathbb{P}(|x+W_n|\leq nr)\cr
&\leq \mathbb{P}(|x|-|W_n|\leq nr)\cr
&= \mathbb{P}(|W_n|\geq |x|-nr),
\end{align}
where the first equality follows from the fact that the branching and motion are independent.
\par
 Since $\mathbb{E}[|X|^{\alpha}]<\infty$ for some $\alpha>2$, by (\ref{7t6ij6y67}) and Lemma \ref{Nagaevlem}, we have
\begin{align}\label{ikuy23}
\int_{|x|>n(r+\delta)}\mathbb{P}_{\delta_x}(Z_n(B(nr))>0)dx&\leq
\int_{|x|>n(r+\delta)}\mathbb{P}(|W_n|\geq |x|-nr)dx\cr
&=n\int_{|y|>r+\delta}\mathbb{P}\left(|W_n|\geq n(|y|-r)\right)dy\cr
&\leq\int_{|y|>r+\delta}\frac{C_1}{n^{\alpha-2}(y-r)^\alpha}dy,
\end{align}
where the right hand side converges to $0$ as $n\to\infty$. Plugging (\ref{ikuy23}) into (\ref{wicm45}), and then using Lemma \ref{extinpro1} (i), to get
$$\limsup_{n\to\infty}\int_\mathbb{R}\mathbb{P}_{\delta_x}(Z_n(B(nr))>0)dx\leq \frac{4(r+\delta)}{\sigma^2}.$$
\par
Plugging above into (\ref{jitbndf12}) and letting $\delta\to0$ give that
\begin{align}
\liminf_{n\to\infty}\mathbb{P}(R_n\geq nr)\geq e^{-\frac{4r}{\sigma^2}}.\nonumber
\end{align}
We have completed the proof of Theorem \ref{thdim1}.
\end{proof}

\section{Proof of Theorem \ref{thdim2}: $m=1$, $d=2$}\label{sectionthdim2}
 The key to prove Theorem \ref{thdim2} is to use the fact that the scaling limits of the branching random walk is a super-Brownian motion. So, we first give a brief introduction to the super-Brownian motion.
\par
Let $M_F(\mathbb{R}^d)$ be the space of finite measures on $\mathbb{R}^d$ equipped with the topology of weak convergence. Let $\psi(u)$ be a function of the form
   $$\psi(u)=a u+b u^2+\int_{(0,\infty)}\left(e^{-ru}-1+ru\right)n(dr),~u\geq0,$$
   where $a\in\mathbb{R}$, $b\geq0$ and $n$ is a $\sigma$-finite measure on $(0,\infty)$ such that $\int_{(0,\infty)}(r\wedge r^2) n(dr)<\infty$. The super-Brownian motion with initial value $\mu\in M_F(\mathbb{R}^d)$ and branching mechanism $\psi$ is an $M_F(\mathbb{R}^d)$-valued process, whose transition probabilities are characterized through their Laplace transforms. For any $\mu\in M_F(\mathbb{R}^d)$ and nonnegative continuous function $\phi(x)$, we have
   \begin{align}\label{ujytas}
  \mathbb{E}_\mu\left[e^{-\int_{\mathbb{R}^d}\phi(x)X_t(dx)}\right]=e^{-\int_{\mathbb{R}^d} u(t,x)\mu(dx)},
   \end{align}
  where $u(t,x)$ is the unique positive solution to the following nonlinear partial differential equation:
  \begin{align}
   \begin{cases}
   \frac{\partial u(t,x)}{\partial t}=\frac{1}{2}\Delta u(t,x)-\psi(u(t,x)),\cr
   u(0,x)=\phi(x).\nonumber
   \end{cases}
   \end{align}
In above, $\Delta u(t,x):=\sum^d_{i=1}\frac{\partial^2 u(t,x)}{\partial x_i^2}$ is the Laplace operator. $\{X_t\}_{t\geq0}$ is called a supercritical (critical, subcritical) super-Brownian motion if $a<0~(=0, >0).$ We refer the reader to \cite{Etheridge}, \cite{perkins}, \cite{LeGall} and \cite{Li} for a more detailed overview to super-Brownian motion.
\par
In this paper, we always consider the critical branching mechanism $\psi(u)=\sigma^2u^2$ (recall that $\sigma^2=\sum^\infty_{k=0}k^2p_k-\left[\sum^\infty_{k=0}kp_k\right]^2$). We first present a lemma concerning the probability of a super-Brownian motion charges the level set $[x,\infty)$. Denote by $\{H_t\}_{t\geq0}$ the historical super Brownian motion; see \cite[p187]{perkins} for a formal definition. Intuitively, $H_t$ keeps track of the histories of all the masses of $X_t$. Let $S(H_t)$ be the closed support of the random measure $H_t$ and $C(\mathbb{R}_+)$ be the space of continuous functions from $\mathbb{R}_+$ to $\mathbb{R}$. From \cite[p195]{perkins}, we have the following:\\
(i) there exist a constant $c>2$ and a random variable $\Delta$ such that $\mathbb{P}_{\delta_0}$ almost surely, for all $t\geq0$,
\begin{align}\label{refevr12sa}
S(H_t)\subset\left\{y(\cdot)\in C(\mathbb{R}_+):|y(r)-y(s)|\leq c|(r-s)\log(r-s)|^{1/2}, \forall r,s>0,|r-s|\leq\Delta\right\};
\end{align}
(ii) there are constants $\rho,~\kappa>0$ depending only on $c$ such that
\begin{align}\label{refevr12saq}
\mathbb{P}_{\delta_0}(\Delta\leq r)\leq \kappa r^{\rho}~\text{for}~r\in[0,1].
\end{align}
\begin{lemma}\label{gjbgk45}Let $\{X_t\}_{t\geq0}$ be an one-dimensional critical super-Brownian motion. Then there exists a constant $C_2>0$ such that for any $x>0$,
$$\mathbb{P}_{\delta_0}(X_1([x,\infty))>0)<C_2e^{-x}.$$
\end{lemma}
\begin{proof}
Let
$$T:=\inf\{s\geq0:X_s([x,\infty))>0\}.$$
Define
\begin{align}
a(x):&=\lfloor e^{x/\rho}\rfloor;\cr
A(t,x):&=\left\{y(\cdot)\in C(\mathbb{R}_+):\sup_{s\in[0,t]}y(s)\geq x\right\}.\nonumber
\end{align}
Because $\lim_{x\to\infty}a(x)=+\infty$, there exists a constant $M$ such that for all $x>M$ and $|r-s|<\frac{2}{a(x)}$,
\begin{align}\label{4tewrw34sw}
c|(r-s)\log(r-s)|^{1/2}<1.
\end{align}
Since under the event $\{\exists s\in[0,1], H_s(A(s,x))>0\}$, we have $T\in(0,1]$, there exists an integer $i\in\{0,1,...,a(x)-1\}$ such that
$T\in(i/a(x),(i+1)/a(x)]$.
\par
For $i\geq1$, if $\{X_t\}_{t\geq0}$ has charged the set $[x-1,\infty)$ at time $(i-1)/a(x)$, then
$$H_{i/a(x)}(A((i-1)/a(x),x-1))>0.$$
Otherwise it has not charged the set $[x-1,\infty)$, then the support process for $\{X_t\}_{t\geq0}$ has to travel a distance of at least $1$ on time interval $[(i-1)/a(x),T]\subset[(i-1)/a(x),(i+1)/a(x)]$.
This, combined with (\ref{refevr12sa}) and (\ref{4tewrw34sw}), entails that for $x>M$,
  $$\Delta<\frac{2}{a(x)}. $$
  \par
  For $i=0$, since $\{X_t\}_{t\geq0}$ has charged the set $[x,\infty)$ on time interval $(0,1/a(x)]$, we have
  $$\Delta<\frac{1}{a(x)}.$$
\par
Putting above together, it follows that
\begin{align}\label{54gmwd2i}
\mathbb{P}_{\delta_0}\left(X_1([x,\infty))>0\right)&\leq\mathbb{P}_{\delta_0}\left(\exists s\in [0,1], X_s([x,\infty))>0\right)\cr
&=\mathbb{P}_{\delta_0}\left(\exists s\in [0,1],H_s(A(s,x))>0\right)\cr
&\leq\sum^{a(x)}_{i=1}\mathbb{P}_{\delta_0}\left(H_{i/a(x)}(A((i-1)/a(x),x-1))>0\right)+\mathbb{P}_{\delta_0}\left(\Delta<\frac{2}{a(x)}\right).
\end{align}
\par
Define $\mathcal{F}_t:=\sigma(H_s,s\in[0,t])$. Let $\{B_t\}_{t\geq0}$ be an one-dimensional Brownian motion started from  the origin. Note that $\{H_t\}_{t\geq0}$ is a time-inhomogeneous Markov process. Write $\mathbb{P}_{{s, u}}(H_{s+t}\in \cdot):=\mathbb{P}(H_{s+t}\in \cdot|H_{s}=u)$ for $s,t\geq 0$ and a measure $u$ on the space of c$\grave{\text{a}}$dl$\grave{\text{a}}$g paths from $[0,s]$ to $\mathbb{R}$; see \cite[p188]{perkins} for a more detailed description. By the Markov property of $\{H_t\}_{t\geq0}$ and \cite[p194,~Lemma III.1.2]{perkins},
\begin{align}\label{65gsw3vf}
&\mathbb{P}_{\delta_0}\left(H_{i/a(x)}(A((i-1)/a(x),x-1))>0\right)\cr
&=1-\mathbb{P}_{\delta_0}\left[\mathbb{P}_{(i-1)/a(x),H_{(i-1)/a(x)}}\left[H_{i/a(x)}(A((i-1)/a(x),x-1))=0|\mathcal{F}_{(i-1)/a(x)}\right]\right]\cr
&=1-\mathbb{E}_{\delta_0}\left[\exp\left\{-2a(x)H_{(i-1)/a(x)}(A((i-1)/a(x),x-1))\right\}\right]\cr
&\leq2a(x)\mathbb{E}_{\delta_0}\left[H_{(i-1)/a(x)}(A((i-1)/a(x),x-1))\right]\cr
&=2a(x)\mathbb{P}\left(\sup_{s\in[0,(i-1)/a(x)]}|B_s|\geq x-1\right),
\end{align}
 where the inequality follows from the fact that $1-e^{-x}\leq x$ and the last equality is because that the mean measure of $H_t$ is the Brownian motion stopped at time $t$ (i.e. one moment formula of $H_t$; see \cite[p191, II.8.5]{perkins}).
 \par
 It is well-known that $\sup_{s\in(0,t)}B_s$ and $|B_t|$ have the same law under $\mathbb{P}_{\delta_0}$. For $i\geq 2$ and $x>2$,
\begin{align}\label{54tgfuh2}
&\mathbb{P}\left(\sup_{s\in[0,(i-1)/a(x)]}|B_s|\geq x-1\right)\cr
&=\mathbb{P}\left(\sup_{s\in[0,(i-1)/a(x)]}B_s\geq(x-1)~\text{or}~\inf_{s\in[0,(i-1)/a(x)]}B_s\leq-(x-1) \right)\cr
&\leq 2\mathbb{P}\left(\sup_{s\in[0,(i-1)/a(x)]}B_s\geq(x-1) \right)\cr
&=2\mathbb{P}\left(|B_{(i-1)/a(x)}|\geq(x-1) \right)\cr
&\leq4\mathbb{P}\left(B_1\geq x-1\right)\cr
&\leq\frac{8}{x\sqrt {2\pi}}e^{-x^2/4},
\end{align}
where the last inequality follows from the following classical estimate for the standard normal
random variable:
$$
\mathbb{P}\left(B_1>x\right)\leq\frac{1}{x\sqrt {2\pi}}e^{-x^2/2}.
$$
\par
Recall that $a(x)=\lfloor e^{x/\rho}\rfloor$. Plugging (\ref{54tgfuh2}), (\ref{65gsw3vf}) into (\ref{54gmwd2i}) yields that there exists a constant $C_2>0$ such that for $x$ large enough,
\begin{align}
\mathbb{P}_{\delta_0}\left(X_1([x,\infty))>0\right)&\leq \frac{16a^2(x)}{x\sqrt {2\pi}}e^{-x^2/4}+\mathbb{P}_{\delta_0}\left(\Delta<\frac{2}{a(x)}\right)\cr
&\leq a^2(x)e^{-x^2/4}+\kappa 2^{\rho}a^{-\rho}(x)\cr
&\leq C_2e^{-|x|},\nonumber
\end{align}
where the second inequality follows from (\ref{refevr12saq}). Thus, the lemma follows.
\end{proof}
Now, we are ready to prove Theorem \ref{thdim2}. The idea of the proof is as follows. For the upper bound, from \cite[p75]{lalley2015} we know that under $\mathbb{P}_{n\delta_0},$
\begin{align}\label{9o86h6t}
\lim_{n\to\infty}\left\{\frac{Z_{\lfloor nt\rfloor}(\sqrt n\cdot)}{n}\right\}_{t\geq0}\overset{\text{weakly}}=\{X_t\}_{t\geq 0},
\end{align}
where $\{X_t\}_{t\geq 0}$ is a two-dimensional critical super-Brownian motion with branching mechanism $\psi(u)=\sigma^2u^2$ and initial value $X_0=\delta_0.$ Thus, the Laplace functional of $Z_n(\sqrt n\cdot)/n$ converges to the Laplace functional of $X_1$, which can be used to deal with the probability $\mathbb{P}_{n\delta_0}\left(Z_n(\sqrt nB(x,r))=0\right)$. By this method, we transform the empty ball problem for branching random walks to the problem for super-Brownian motions which have been studied in \cite{xz21}. For the lower bound, we mainly use the fact that the condition maximum of a critical branching random walk converges in distribution to the maximum of some measure valued process $\{Y_t\}_{t\geq0}$; see \cite[Theorem 3]{lalley2015}. Moreover, $\{Y_t\}_{t\geq0}$ is closely related to the super-Brownian motion $\{X_t\}_{t\geq 0}$. Thus, we can use Lemma \ref{gjbgk45} to obtain the desired lower bound.
\begin{proof}
\textbf{Upper bound.}
Similarly to (\ref{jitbndf12}), we have
\begin{align}\label{jitbndf12dim2}
\mathbb{P}(R_n\geq \sqrt nr)&=\exp\left\{-\int_{\mathbb{R}^2}\mathbb{P}_{\delta_x}(Z_n(B(\sqrt nr))>0)dx\right\}\cr
&=\exp\left\{-\int_{\mathbb{R}^2}n\mathbb{P}_{\delta_0}(Z_n(\sqrt nB(x,r))>0)dx\right\}.
\end{align}
Observe that
\begin{align}
\mathbb{P}_{n\delta_0}\left(\frac{1}{n}Z_n(\sqrt nB(x,r))=0\right)=\left[1-\mathbb{P}_{\delta_0}\left(Z_n(\sqrt nB(x,r))>0\right)\right]^n.\nonumber
\end{align}
Thus,
\begin{align}\label{yhgbr4f}
&\mathbb{P}_{\delta_0}\left(Z_n(\sqrt nB(x,r))>0\right)\cr
&=1-\exp\left\{\frac{1}{n}\log\mathbb{P}_{n\delta_0}\left(\frac{1}{n}Z_n(\sqrt nB(x,r))=0\right)\right\}\cr
&\geq -\frac{1}{n}\log\mathbb{P}_{n\delta_0}\left(\frac{1}{n}Z_n(\sqrt nB(x,r))=0\right)\left[1+\frac{1}{2n}\log\mathbb{P}_{n\delta_0}\left(Z_n(\sqrt nB(x,r))=0\right)\right]\cr
&= -\frac{1}{n}\log\mathbb{P}_{n\delta_0}\left(\frac{1}{n}Z_n(\sqrt nB(x,r))=0\right)\left[1+\frac{1}{2}\log\mathbb{P}_{\delta_0}\left(Z_n(\sqrt nB(x,r))=0\right)\right]\cr
&\geq -\frac{1}{n}\log\mathbb{P}_{n\delta_0}\left(\frac{1}{n}Z_n(\sqrt nB(x,r))=0\right)\left[1+\frac{1}{2}\log\mathbb{P}_{\delta_0}\left(|Z_n|=0\right)\right],
\end{align}
where the first inequality follows from the fact that $1-e^x\geq -x(1+x/2)$ for $x<0$. By (\ref{9o86h6t}), we have
\begin{align} \label{rfdwexc34}
&\liminf_{n\to\infty}-\log\mathbb{P}_{n\delta_0}\left(\frac{1}{n}Z_n(\sqrt nB(x,r))=0\right)\cr
&=\liminf_{n\to\infty}\lim_{\theta\to\infty}-\log\mathbb{E}_{n\delta_0}\left[e^{-\frac{1}{n}Z_n(\theta\sqrt nB(x,r))}\right]\cr
&\geq\lim_{\theta\to\infty}\liminf_{n\to\infty}-\log\mathbb{E}_{n\delta_0}\left[e^{-\frac{1}{n}Z_n(\theta\sqrt nB(x,r))}\right]\cr
&=\lim_{\theta\to\infty}-\log\mathbb{E}_{\delta_0}\left[e^{-X_1(\theta B(x,r))}\right]\cr
&=-\log\mathbb{P}_{\delta_0}\left(X_1(B(x,r))=0\right),
\end{align}
where the inequality follows from the fact that $-\log\mathbb{E}_{n\delta_0}\left[e^{-\frac{1}{n}Z_n(\theta\sqrt nB(x,r))}\right]$ is increasing w.r.t. $\theta$. Plugging (\ref{rfdwexc34}) and (\ref{yhgbr4f}) into (\ref{jitbndf12dim2}) yields that
\begin{align}
&\limsup_{n\to\infty}\mathbb{P}(R_n\geq \sqrt nr)\cr
&=\exp\left\{-\liminf_{n\to\infty}\int_{\mathbb{R}^2}n\mathbb{P}_{\delta_0}(Z_n(\sqrt nB(x,r))>0)dx\right\}\cr
&\leq\exp\left\{-\int_{\mathbb{R}^2}\liminf_{n\to\infty}n\mathbb{P}_{\delta_0}(Z_n(\sqrt nB(x,r))>0)dx\right\}\cr
&\leq\exp\left\{-\int_{\mathbb{R}^2}\liminf_{n\to\infty} -\log\mathbb{P}_{n\delta_0}\left(\frac{1}{n}Z_n(\sqrt nB(x,r))=0\right)\left[1+\frac{1}{2}\log\mathbb{P}_{\delta_0}\left(|Z_n|=0\right)\right]dx\right\}\cr
&\leq\exp\left\{\int_{\mathbb{R}^2}\log\mathbb{P}_{\delta_0}\left(X_1(B(x,r))=0\right)dx\right\},
\end{align}
where the first inequality follows from Fatou's Lemma.\\

\smallskip
\noindent\textbf{Lower bound.} Fix $\delta>0$. Note that
\begin{align}\label{fbfbfgjg34}
&\mathbb{P}(R_n\geq \sqrt nr)\cr
&=\exp\left\{-\int_{\mathbb{R}^2}n\mathbb{P}_{\delta_0}(Z_n(\sqrt nB(x,r))>0)dx\right\}\cr
&\geq\exp\left\{-\int_{|x|\leq(1+\delta)r}n\mathbb{P}_{\delta_0}(|Z_n|>0)dx-\int_{|x|>(1+\delta)r}n\mathbb{P}_{\delta_0}(Z_n(\sqrt nB(x,r))>0)dx\right\}.
\end{align}
The first term on the right hand side of (\ref{fbfbfgjg34}) can be well-handled by Lemma \ref{extinpro1}.
\par
 In the following, we are going the deal with the second term. For a vector $x$, denote by $x^{(i)}$ the $i-$th component of $x$. Let $M_n:=\max\{|S_u|:u\in Z_n\}$, $M'_n:=\max\{|S^{(1)}_u|:u\in Z_n\}$, $M''_n:=\max\{S^{(1)}_u: u\in Z_n\}$. Without loss of generality, we assume that $X^{(1)}$ and $X^{(2)}$ have some common symmetric distribution. If not the case, one can make slight modifications to let the following arguments carry through. Note that for $|x|>(1+\delta)r$,
\begin{align}
\mathbb{P}_{\delta_0}(Z_n(\sqrt nB(x,r))>0)&\leq \mathbb{P}_{\delta_0}(M_n>\sqrt n(|x|-r))\cr
&\leq 2\mathbb{P}_{\delta_0}(M'_n>\sqrt n(|x|-r))\cr
&\leq 4\mathbb{P}_{\delta_0}(M''_n>\sqrt n(|x|-r)),\nonumber
\end{align}
where the second inequality follows from the fact that $X^{(1)}$ and $X^{(2)}$ have the same law and the last inequality is because that the step size is symmetric. So, it yields that
\begin{align}\label{54g4d23d}
&\int_{|x|>(1+\delta)r}n\mathbb{P}_{\delta_0}(Z_n(\sqrt nB(x,r))>0)dx\cr
&\leq 4\int_{|x|>(1+\delta)r} n\mathbb{P}_{\delta_0}(M''_n>\sqrt n(|x|-r))dx\cr
&\leq 4n\mathbb{P}_{\delta_0}(|Z_n|>0)\int_{|x|>(1+\delta)r} \mathbb{P}_{\delta_0}\left(M''_n>\sqrt n(|x|-r)\Big{|}|Z_n|>0\right)dx.
\end{align}
From \cite[Corollary 4]{lalley2015}, we have
\begin{align}\label{4f3w2e23}
\lim_{n\to\infty}\mathbb{P}_{\delta_0}(M''_n\geq\sqrt n(|x|-r)||Z_n|>0)=\mathbb{P}_{\delta_0}(Y_1([|x|-r,\infty))>0),
\end{align}
where $\{Y_t\}_{t\geq 0}$ is an $1$-dimensional measure-valued process. Moreover, $\{Y_t\}_{t\geq 0}$ satisfies
\begin{align}\label{4f3w2e2323}
\sum^N_{i=1}Y^i_t\overset{\text{law}}= X_t,
\end{align}
where $N$ is a Poisson random variable with mean $1$ and $Y^i_t$, $i\geq1$ are independent copies of $Y_t$. According to \cite[Theorem 1.1]{Kesten1995}, there exists a constant $C_3>0$ such that
\begin{align}
\mathbb{P}_{\delta_0}(M''_n\geq\sqrt n(|x|-r)||Z_n|>0)\leq C_3\frac{1+(\log(|x|-r))^2}{(|x|-r)^4},\nonumber
\end{align}
which is integrable on $\{|x|>(1+\delta)r\}\subset\mathbb{R}^2$. Hence, one can use Fatou lemma, (\ref{54g4d23d}) and (\ref{4f3w2e23}) to obtain that
\begin{align}\label{yngng34b}
&\limsup_{n\to\infty}\int_{|x|>(1+\delta)r}n\mathbb{P}_{\delta_0}(Z_n(\sqrt nB(x,r))>0)dx\cr
&\leq \frac{8}{\sigma^2}\int_{|x|>(1+\delta)r}\mathbb{P}_{\delta_0}\left(Y_1([|x|-r))>0\right)dx.
\end{align}
\par
On the other hand, by (\ref{4f3w2e2323}),
\begin{align}\label{fbgfgb6}
\mathbb{P}_{\delta_0}\left(X_1([|x|-r,\infty))>0\right)&=1-\mathbb{P}_{\delta_0}\left(\sum^N_{i=1}Y^i_1([|x|-r,\infty))=0\right)\cr
&=1-\sum_{k\geq0}\frac{e^{-1}}{k!}\left(\mathbb{P}_{\delta_0}\left(Y_1([|x|-r,\infty))=0\right)\right)^k\cr
&=1-e^{-\mathbb{P}_{\delta_0}\left(Y_1([|x|-r,\infty))>0\right)}.
\end{align}
Note that there exists $C_\delta\in(0,1)$ such that
$$-\log(1-x)<(1+\delta)x,~x\in(0,C_{\delta}).$$ According to Lemma \ref{gjbgk45}, there exists $C'_{\delta}>0$ such that for $|x|\geq (1+\delta)r$ and $r>C'_{\delta}$,
\begin{align}
\mathbb{P}_{\delta_0}\left(X_1([|x|-r,\infty))>0\right)&\leq\mathbb{P}_{\delta_0}\left(X_1([\delta r,\infty))>0\right)<C_{\delta}.\nonumber
\end{align}
This, combined with (\ref{fbgfgb6}), yields that for $|x|\geq (1+\delta)r$ and $r>C'_{\delta}$,
\begin{align}
\mathbb{P}_{\delta_0}\left(Y_1([|x|-r,\infty))>0\right)&=-\log\left[1-\mathbb{P}_{\delta_0}\left(X_1([|x|-r,\infty))>0\right)\right]\cr
&\leq (1+\delta)\mathbb{P}_{\delta_0}\left(X_1([|x|-r,\infty))>0\right)\cr
&\leq C_2(1+\delta) e^{-(|x|-r)},\nonumber
\end{align}
where the last inequality follows from Lemma \ref{gjbgk45}. Plugging above into (\ref{yngng34b}) yields that
\begin{align}
&\limsup_{n\to\infty}\int_{|x|>(1+\delta)r}n\mathbb{P}_{\delta_0}(Z_n(\sqrt nB(x,r))>0)dx\cr
&\leq\int_{|x|>(1+\delta)r} \frac{8}{\sigma^2}C_2(1+\delta) e^{-(|x|-r)}dx\cr
&= \frac{8}{\sigma^2}C_2(1+\delta)\int^{\infty}_{\delta r} (u+r)e^{-u}du\cr
&\leq  \frac{8}{\sigma^2}C_2(1+\delta)(1+1/\delta)e^{-\delta r/2}.\nonumber
\end{align}
\par
The above inequality, together with (\ref{fbfbfgjg34}) and Lemma \ref{extinpro1}, yields that for any $\delta>0$,
\begin{align}
\liminf_{n\to\infty}\mathbb{P}(R_n\geq \sqrt nr)\geq e^{- \frac{4\pi(1+\delta)^2r^2}{\sigma^2}- \frac{8}{\sigma^2}C_2(1+\delta)(1+1/\delta)e^{-\delta r/2}}.\nonumber
\end{align}
Thus, the desired lower bound follows by letting $C(\delta):=\frac{8}{\sigma^2}C_2(1+\delta)(1+1/\delta)$.
\end{proof}

\section{Proof of Theorem \ref{thdim3}: $m=1$, $d\geq3$}\label{secthdimddayu3}

The first result below tells us that for any $d\geq1$, $R_n$ converges in law. The result is universal for any critical and subcritical branching random walk.
\begin{proposition}\label{pro1}Assume $m\leq 1$ and $d\geq 1$. Then, there exists a function $F_d(r)\in[0,+\infty]$ such that for  $r>0$ ,
$$
\lim_{n\to\infty}\mathbb{P}\left(R_n\geq r\right)=e^{-F_d(r)}\in[0,1].
$$
\end{proposition}
\begin{proof} Similarly to (\ref{jitbndf12}), we have
\begin{align}\label{embr1}
\mathbb{P}(R_n\geq r)&=\mathbb{P}\left(Z_n(B(r))=0\right)\cr
&=e^{-\int_{\mathbb{R}^d} 1-\mathbb{P}_{\delta_{x}}(Z_n(B(r))=0)dx}.
\end{align}
For $\theta>0$, write $$u_{\theta}(k,x):=-\log\mathbb{E}_{\delta_x}\left[e^{-\theta Z_k(B(r))}\right],~k\geq1.$$
Since $u_{\theta}(k,x)$ is increasing w.r.t. $\theta$, let $u(k,x):=\lim_{\theta\to\infty}u_{\theta}(k,x)$. Observe that
\begin{align}\label{embr2}
\mathbb{P}_{\delta_x}(Z_n(B(r))=0)&=\lim_{\theta\to\infty}\mathbb{E}_{\delta_x}\left[e^{-\theta Z_n(B(r))}\right]\cr
&=\lim_{\theta\to\infty}e^{-u_{\theta}(n,x)}\cr
&=e^{-u(n,x)}.
\end{align}
This, combined with (\ref{embr1}), gives
\begin{align}
\mathbb{P}(R_n\geq r)=\exp\left\{-\int_{\mathbb{R}^d}1-e^{-u(n,x)}dx\right\}.\nonumber
\end{align}
\par
Note that for all $x\in\mathbb{R}^d$ and $n\geq 1$,
\begin{align}\label{fghwra}
u(n,x)&=-\log\mathbb{P}_{\delta_x}(Z_n(B(r))=0)\cr
&\leq -\log\mathbb{P}_{\delta_0}(|Z_n|=0)\cr
&=:c_n.
\end{align}
Since $y-\frac{y^2}{2}<1-e^{-y}<y$ for $y>0$, we have
\begin{align}\label{dsgwerq}
\exp\left\{-\int_{\mathbb{R}^d}u(n,x)dx\right\}&<\mathbb{P}(R_n\geq r)\cr
&<\exp\left\{-\int_{\mathbb{R}^d}u(n,x)-\frac{u^2(n,x)}{2}dx\right\}\cr
&\leq\exp\left\{-\left(1-\frac{c_n}{2}\right)\int_{\mathbb{R}^d}u(n,x)dx\right\}.
\end{align}
\par
Recall that $f(s)=\sum_{k\geq 0}p_ks^k$. By the branching property,
\begin{align}
e^{-u_{\theta}(n,x)}&=\mathbb{E}_{\delta_x}\left[e^{-\theta Z_n(B(r))}\right]\cr
&=\mathbb{E}_{\delta_x}\left[e^{-\theta \sum_{u\in Z_1}Z^{\delta_{S_u}}_{n-1}(B(r))}\right]\cr
&=\sum_{k\geq 0}p_k\mathbb{E}^k_{\delta_x}\left[e^{-\theta Z^{\delta_{W_1}}_{n-1}(B(r))}\right]\cr
&=\sum_{k\geq0}p_k\mathbb{E}^k_{x}\left[e^{-u_{\theta}(n-1,W_1)}\right]\cr
&=f\left(\mathbb{E}_{x}\left[e^{-u_{\theta}(n-1,W_1)}\right]\right)\cr
&\geq \mathbb{E}_{x}\left[e^{-u_{\theta}(n-1,W_1)}\right]\cr
&\geq e^{-\mathbb{E}_{x}\left[u_{\theta}(n-1,W_1)\right]},\nonumber
\end{align}
where the first inequality follows by the fact that $f(s)\geq s$ (see \cite[p4]{r1}) and the second inequality follows from the Jensen inequality. Thus, this yields
$$
u_{\theta}(n,x)\leq\mathbb{E}_{x}\left[u_{\theta}(n-1,W_1)\right]
$$
By Fubini's theorem, it follows that
\begin{align}
\int_{\mathbb{R}^d}u_{\theta}(n,x)dx&\leq\int_{\mathbb{R}^d}\mathbb{E}_{x}\left[u_{\theta}(n-1,W_1)\right]dx\cr
&=\mathbb{E}_{0}\left[\int_{\mathbb{R}^d}u_{\theta}(n-1,x+X)dx\right]\cr
&=\int_{\mathbb{R}^d}u_{\theta}(n-1,x)dx.\nonumber
\end{align}
Since $u_{\theta}(n,x)$ is increasing w.r.t. $\theta$, by L\'evy's monotone convergence lemma, the above yields
$$
\int_{\mathbb{R}^d}u(n,x)dx\leq\int_{\mathbb{R}^d}u(n-1,x)dx.
$$
Therefore, by the monotone convergence theorem $\lim_{n\to\infty} e^{-\int_{\mathbb{R}^d}u(n,x)dx}$ exists.
 \par
On the other hand, using the assumption that the branching process is non-supercritical,  we have $\lim_{n\to\infty}c_n=0$ (see (\ref{fghwra})). This, together with (\ref{dsgwerq}), entails that
$$\lim_{n\to\infty}\mathbb{P}\left(R_n\geq r\right)=\lim_{n\to\infty} e^{-\int_{\mathbb{R}^d}u(n,x)dx}~\text{exists}.$$

\end{proof}

Now, we are ready to prove Theorem \ref{thdim3}: if $m=1$, $\sigma^2<\infty$, $d\geq3$ and $\mathbb{E}[|X|^3]<\infty$, then

$$\lim_{n\to\infty}\mathbb{P}\left(R_n\geq r\right)=e^{-F_d(r)}\in(0,1).$$
Moreover,
\begin{align}\label{4rtrer3}
\frac{v_d(r)}{1+\sigma^2 C_d(r)r^2}\leq F_d(r)\leq v_d(r).
\end{align}
\par
By Proposition \ref{pro1}, it suffices to show (\ref{4rtrer3}). The idea of the proof is as follows. By the second moment method, to give an upper bound of $\mathbb{P}\left(R_n\geq r\right)$, it suffices to deal with the probability that two correlated random walks are in some fixed ball at time $n$. This probability can be well-handled by the Berry-Esseen inequality. We obtain the lower bound by the one moment estimation.
\begin{proof}
\textbf{Upper bound.}
Similarly to (\ref{jitbndf12}), one sees that
\begin{align}\label{dfev23}
\mathbb{P}(R_n\geq r)&=\exp\left\{-\int_{\mathbb{R}^d}\mathbb{P}_{\delta_x}(Z_n(B(r))>0)dx\right\}.
\end{align}
By the Paly-Zygmund inequality, we have
\begin{align}\label{65y4r3}
\mathbb{P}_{\delta_x}(Z_n(B(r))>0)&\geq\frac{\mathbb{E}_{\delta_x}^2[Z_n(B(r))]}{\mathbb{E}_{\delta_x}[Z^2_n(B(r))]}\cr
&=\frac{\mathbb{P}(|x+W_n|<r)^2}{\mathbb{E}_{\delta_x}[Z^2_n(B(r))]},
\end{align}
where the equality follows from the independent of the branching and the motion. It is simple to see that
\begin{align}\label{egert4}
\mathbb{E}_{\delta_x}[Z^2_n(B(r))]&=\mathbb{E}_{\delta_x}\Big[\Big(\sum_{u\in Z_n}\ind_{\{S_u\in B(r)\}}\Big)\Big(\sum_{v\in Z_n}\ind_{\{S_v\in B(r)\}}\Big)\Big]\cr
&=\mathbb{E}_{\delta_x}\Bigg[\sum_{u\in Z_n}\ind_{\{S_u\in B(r)\}}+\sum_{u, v\in Z_n\atop u\neq v}\ind_{\{S_u\in B(r)\}}\ind_{\{S_v\in B(r)\}}\Bigg]\cr
&=\mathbb{P}(|x+W_n|<r)+\mathbb{E}_{\delta_x}\Bigg[\sum_{u, v\in Z_n\atop u\neq v}\ind_{\{S_u\in B(r)\}}\ind_{\{S_v\in B(r)\}}\Bigg].
\end{align}
\par
In the next, we are going to calculate $\mathbb{E}_{\delta_x}\Bigg[\sum\limits_{u, v\in Z_n\atop u\neq v}\ind_{\{S_u\in B(r)\}}\ind_{\{S_v\in B(r)\}}\Bigg].$ For $j$, $i\geq0$ and~$v\in Z_i$, let
\begin{align}
&Z^v_j:=\{u\in Z_{i+j}: u~\text{is~a~progeny~of}~v\}.\nonumber
\end{align}
The notation $u\sim_{k}v$ means the most recent common ancestor of $u$ and $v$ is in generation $k$.
Let $\mathcal{F}_n:=\sigma(|Z_i|,i\leq n)$. Recall that $W_n=X_1+X_2+....+X_n$, where $X_i$, $i\geq1$ are independent copies of the step size $X$. For $0\leq k\leq n-1$, let
$$W_n(k):=X_1+X_2+...+X_k+X'_{k+1}+X'_{k+2}+...+X'_{n},$$
where $X'_i$, $i\geq1$ are independent copies of the step size $X$ and independent of $X_i$, $i\geq1$. Since the branching and spatial motion are independent, we have
\begin{align}\label{yten12}
&\mathbb{E}_{\delta_x}\Bigg[\sum_{u, v\in Z_n\atop u\neq v}\ind_{\{S_u\in B(r)\}}\ind_{\{S_v\in B(r)\}}\Bigg]\cr
&=\mathbb{E}_{\delta_x}\Bigg[\sum^{n-1}_{k=0}\sum_{u, v\in Z_n\atop u\neq v,u\sim_{k}v}\ind_{\{S_u\in B(r)\}}\ind_{\{S_v\in B(r)\}}\Bigg]\cr
&=\sum^{n-1}_{k=0}\mathbb{E}_{\delta_x}\Bigg[\sum_{u, v\in Z_n\atop u\neq v,u\sim_{k}v}\ind_{\{S_u\in B(r)\}}\ind_{\{S_v\in B(r)\}}\Bigg]\cr
&=\sum^{n-1}_{k=0}\mathbb{E}_{\delta_x}\left[\sum_{\varsigma\in Z_k}\sum_{w,w'\in Z^{\varsigma}_1, w\neq w'\atop
u\in Z^w_{n-k-1},v\in Z^{w'}_{n-k-1}}\ind_{\{S_u\in B(r)\}}\ind_{\{S_v\in B(r)\}}\right]\cr
&=\sum^{n-1}_{k=0}\mathbb{E}_{\delta_x}\left[\mathbb{E}_{\delta_x}\left[\sum_{\varsigma\in Z_k}\sum_{w,w'\in Z^{\varsigma}_1, w\neq w'\atop
u\in Z^w_{n-k-1},v\in Z^{w'}_{n-k-1}} \ind_{\{S_u\in B(r)\}}\ind_{\{S_v\in B(r)\}}\Bigg|\mathcal{F}_n\right]\right]\cr
&=\sum^{n-1}_{k=0}\mathbb{P}_{\delta_x}\left(W_n\in B(r),W_n(k)\in B(r)\right)\mathbb{E}_{\delta_x}\left[\sum_{\varsigma\in Z_k}\sum_{w,w'\in Z^{\varsigma}_1, w\neq w'\atop
u\in Z^w_{n-k-1},v\in Z^{w'}_{n-k-1}}1\right]\cr
&=\sum^{n-1}_{k=0}\mathbb{P}_{\delta_x}\left(W_n\in B(r),W_n(k)\in B(r)\right)\mathbb{E}_{\delta_x}\left[\sum_{\varsigma\in Z_k}
(|Z^{\varsigma}_1|-1)|Z^{\varsigma}_1||Z^w_{n-k-1}||Z^{w'}_{n-k-1}|\right]\cr
&=\left[\mathbb{E}_{\delta_0}\left[|Z_1|^2\right]-1\right]\sum^{n-1}_{k=0}\mathbb{P}\left(W_n\in B(-x,r),W_n(k)\in B(-x,r)\right),
\end{align}
where the last inequality follows from the fact that $|Z_k|,~|Z^{\varsigma}_1|,~|Z^w_{n-k-1}|~\text{and}~|Z^{w'}_{n-k-1}|$ are independent. It suffices to give an upper bound of the second factor of the right hand side of (\ref{yten12}).

Let $\{\xi_n\}_{n\geq 1}$ be an one-dimensional random walk satisfying $\mathbb{E}[\xi_1]=0,~\mathbb{E}\left[|\xi_1|^3\right]<
\infty$. Recall that $\{B_t\}_{t\geq0}$ is a standard Brownian motion. From \cite[p542]{Feller}, we have the following Berry-Esseen inequality. For any interval $A\subset \mathbb{R}$ and $n\geq 1$,
\begin{align}\label{revrd34}
\left|\mathbb{P}\left(\xi_n/\sqrt n\in A\right)-\mathbb{P}(B_1\in A)\right|\leq\frac{6\mathbb{E}\left[|\xi_1|^3\right]}{\mathbb{E}\left[|\xi_1|^2\right]^{3/2}\sqrt n}.
\end{align}
 Recall that $x^{(i)},~1\leq i\leq d$ stands for the $i$-th component of a $d$-dimensional vector $x$. Let $C_4:=6\mathbb{E}[|X^{(1)}|^3]/(\mathbb{E}[|X^{(1)}|^2]^{3/2})$. Then by (\ref{revrd34}), for any $0\leq k\leq n-1$, $x\in\mathbb{R}^d$ and $r>0$,

\begin{align}
\mathbb{P}\left(W_{n-k}\in B(-x,r)\right)&=\mathbb{P}\left(\left|\frac{W_{n-k}+x}{\sqrt{ n-k}}\right|< \frac{r}{\sqrt {n-k}}\right)\cr
&\leq\prod_{i=1}^d\mathbb{P}\left(\left|\frac{W^{(i)}_{n-k}+x^{(i)}}{\sqrt{ n-k}}\right|< \frac{r}{\sqrt {n-k}}\right)\cr
&\leq \prod_{i=1}^d\left[\mathbb{P}\left(\left|B_1+\frac{x^{(i)}}{\sqrt {n-k}}\right|< \frac{r}{\sqrt {n-k}}\right)+\frac{C_4}{\sqrt {n-k}}\right]\cr
&\leq\prod_{i=1}^d\left[\int_{|y+x^{(i)}/\sqrt {n-k}|< \frac{r}{\sqrt {n-k}}}\frac{1}{\sqrt {2\pi}}e^{-y^2/2}dy+\frac{C_4}{\sqrt {n-k}}\right]\cr
&\leq\frac{(r+C_4)^d}{(n-k)^{d/2}}.\nonumber
\end{align}
Thus,
\begin{align}
\mathbb{P}\left(W_n\in B(-x,r),W_n(k)\in B(-x,r)\right)&=\int_{\mathbb{R}^d}\mathbb{P}^2\left(W_{n-k}\in B(-x-y,r)\right)\mathbb{P}\left(W_k\in dy\right)\cr
&\leq\left[\frac{(r+C_4)^d}{(n-k)^{d/2}}\wedge1\right]\mathbb{P}\left(W_n\in B(-x,r)\right).\nonumber
\end{align}
Hence,
\begin{align}\label{yh6}
&\sum^{n-1}_{k=0}\mathbb{P}\left(W_n\in B(-x,r),W_n(k)\in B(-x,r)\right)\cr
&\leq \mathbb{P}\left(W_n\in B(-x,r)\right)\left[\sum_{0\leq k< n-r^2}\frac{(r+C_4)^d}{(n-k)^{d/2}}+\sum_{n-r^2\leq k\leq n-1}1\right]\cr
&\leq\mathbb{P}\left(W_n\in B(-x,r)\right)\left[(r+C_4)^d\int^{n-r^2}_{0}\frac{1}{(n-u)^{d/2}}du+r^2\right]\cr
&\leq\mathbb{P}\left(W_n\in B(-x,r)\right)\left[\frac{2}{d-2}(r+C_4)^dr^{2-d}+r^2\right]\cr
&= C_d(r)r^2\mathbb{P}\left(W_n\in B(-x,r)\right),
\end{align}
where $C_d(r):=\frac{2(r+C_4)^d}{(d-2)r^d}+1$. Plugging (\ref{yh6}) and (\ref{yten12}) into (\ref{egert4}) yields that
$$ \mathbb{E}_{\delta_x}[Z^2_n(B(r))]\leq \mathbb{P}(|x+W_n|<r)+\sigma^2 C_d(r)r^2\mathbb{P}\left(|x+W_n|<r\right).$$
Therefore, by (\ref{65y4r3})
\begin{align}
\mathbb{P}_{\delta_x}(Z_n(B(r))>0)
&\geq\frac{\mathbb{P}(|x+W_n|<r)^2}{\mathbb{P}(|x+W_n|<r)+\sigma^2 C_d(r)r^2\mathbb{P}\left(|x+W_n|<r\right)}\cr
&= \frac{\mathbb{P}(|x+W_n|<r)}{1+\sigma^2 C_d(r)r^2}.\nonumber
\end{align}
Thus, by Fubini's theorem, it follows that
\begin{align}
\int_{\mathbb{R}^d}\mathbb{P}_{\delta_x}(Z_n(B(r))>0)dx&\geq \frac{1}{1+\sigma^2 C_d(r)r^2}\int_{\mathbb{R}^d}\mathbb{P}(|x+W_n|<r)dx\cr
&\geq \frac{1}{1+\sigma^2 C_d(r)r^2}\int_{\mathbb{R}^d}\int_{|x+y|<r}\mathbb{P}(W_n\in dy)dx\cr
&= \frac{1}{1+\sigma^2 C_d(r)r^2}\int_{\mathbb{R}^d}\int_{|x+y|<r}dx\mathbb{P}(W_n\in dy)\cr
&= \frac{v_d(r)}{1+\sigma^2 C_d(r)r^2}.\nonumber
\end{align}
Hence, by (\ref{dfev23}), we obtain
$$
\limsup_{n\to\infty}\mathbb{P}(R_n\geq r)\leq e^{-\frac{v_d(r)}{1+\sigma^2 C_d(r)r^2}}.
$$

\smallskip
\noindent\textbf{Lower bound.}
By Markov inequality, it is easy to see that
\begin{align}
\int_{\mathbb{R}^d}\mathbb{P}_{\delta_x}(Z_n(B(r))>0)dx&\leq\int_{\mathbb{R}^d}\mathbb{E}_{\delta_x}[Z_n(B(r))]dx\cr
&=\int_{\mathbb{R}^d}\mathbb{P}\left(|W_n+x|< r\right)dx\cr
&=\int_{\mathbb{R}^d}\int_{\mathbb{R}^d}\ind_{\{|y+x|< r\}}\mathbb{P}(W_n\in dy)dx\cr
&=\int_{\mathbb{R}^d}\int_{\mathbb{R}^d}\ind_{\{|y+x|< r\}}dx\mathbb{P}(W_n\in dy)\cr
&=v_d(r),\nonumber
\end{align}
where in the second last equality we use Fubini's theorem. Thus, by (\ref{dfev23}),
$$
\liminf_{n\to\infty}\mathbb{P}(R_n\geq r)\geq e^{-v_d(r)}.
$$
\end{proof}
\section{Proof of Theorem \ref{thdimsgimawuqiong}: $m=1$, $\sigma^2=\infty$}\label{secsigmawuqiong}
The proof of this theorem is similar to that for Theorem \ref{thdim1} in spirit.
\begin{proof}
\textbf{Upper bound.}
Recall that in Theorem  \ref{thdimsgimawuqiong}, we have defined
$$b_n=\Big[nL\big(\mathbb{P}_{\delta_0}(|Z_n|>0)\big)\Big]^{\frac{1}{\beta d}}.$$
By Lemma \ref{extinpro1},
\begin{align}\label{extinpro1sigma}
\lim_{n\to\infty}\mathbb{P}_{\delta_0}(|Z_n|>0)b^{d}_n={(\beta^{-1})}^{\beta^{-1}}.
\end{align}
Similarly to (\ref{jitbndf12}), we have
\begin{align}\label{tgbtsg4}
\mathbb{P}(R_n\geq b_nr)&=\exp\left\{-\int_{\mathbb{R}^d}\mathbb{P}_{\delta_x}(Z_n(B(b_nr))>0)dx\right\}.
\end{align}
Since $L(\cdot)$ varies slowly at $0$, according to \cite[p277, Lemma 2]{Feller}, for any $\varepsilon>0$ and $n$ large enough,
\begin{align}\label{4tfs3e3}
(\mathbb{P}_{\delta_0}(|Z_n|>0))^{\varepsilon}<L\big(\mathbb{P}_{\delta_0}(|Z_n|>0)\big)<(\mathbb{P}_{\delta_0}(|Z_n|>0))^{-\varepsilon}.
\end{align}
Thus, by (\ref{extinpro1sigma}), for $\varepsilon\in\left(0,\beta \right)$ and $n$ large enough,
\begin{align}
\varepsilon&\leq\mathbb{P}_{\delta_0}(|Z_n|>0)\left[n(\mathbb{P}_{\delta_0}(|Z_n|>0))^{-\varepsilon}\right]^{\frac{1}{\beta}}\cr
&=(\mathbb{P}_{\delta_0}(|Z_n|>0))^{1-\frac{\varepsilon}{\beta}}n^{\frac{1}{\beta}}\nonumber.
\end{align}
This, together with (\ref{4tfs3e3}), implies
\begin{align}\label{efr45v}
b_n&=\Big[nL\big(\mathbb{P}_{\delta_0}(|Z_n|>0)\big)\Big]^{\frac{1}{\beta d}}\cr
&\geq n^{\frac{1}{\beta d}}(\mathbb{P}_{\delta_0}(|Z_n|>0))^{\frac{\varepsilon}{\beta d}}\cr
&\geq n^{\frac{1}{\beta d}}(\varepsilon n^{-\frac{1}{\beta}})^{\frac{\varepsilon}{(\beta-\varepsilon) d}}.
\end{align}
Since $d<\frac{1}{\beta}$ and $\varepsilon$ can be arbitrary small, (\ref{efr45v}) yields
\begin{align}\label{354rde}
\lim_{n\to\infty}\frac{n}{b_n}=0.
\end{align}
Fix $\delta\in(0,r)$. By the law of large numbers,
\begin{align}\label{56yyq}
\lim_{n\to\infty}\mathbb{P}\left(\frac{|W_n|}{b_n}\leq\delta\right)
=\lim_{n\to\infty}\mathbb{P}\left(\frac{|W_n|}{n}\frac{n}{b_n}\leq\delta\right)=1.
\end{align}
Condition on the event $\{|Z_n|>0\}$, let $u$ be any given particle at time $n$ (for example, one can label particles through the genealogical structure; see \cite[p13]{zhan}). Similarly to (\ref{tg3452}), we have
\begin{align}\label{54tfrf2}
&\int_{\mathbb{R}^d}\mathbb{P}_{\delta_x}\left(Z_n(B(b_nr))>0\right)dx\cr
&\geq\int_{\mathbb{R}^d}\mathbb{P}_{\delta_{x}}(|Z_n|>0)\mathbb{P}_{x}(S_u\in B(b_nr))dx\cr
&\geq\int_{|y|\leq r-\delta}b^d_n\mathbb{P}_{\delta_{b_ny}}(|Z_n|>0)\mathbb{P}(|W_n+b_ny|< b_nr)dy\cr
&\geq\mathbb{P}_{\delta_0}(|Z_n|>0)b^d_nv_d(r-\delta)\mathbb{P}\left(\frac{|W_n|}{b_n}<\delta\right).
\end{align}
This, combined with (\ref{extinpro1sigma}) and (\ref{56yyq}), yields

$$\liminf_{n\to\infty}\int_{\mathbb{R}^d}\mathbb{P}_{\delta_x}\left(Z_n(B(b_nr))>0\right)dx\geq v_d(r-\delta)(\beta^{-1})^{\beta^{-1}}.$$
By letting $\delta\to0$ and (\ref{tgbtsg4}), the above yields

$$\limsup_{n\to\infty}\mathbb{P}\left(\frac{R_n}{b_n}\geq r\right)\leq e^{-v_d(r)(\beta^{-1})^{\beta^{-1}}}.$$

\smallskip
\noindent\textbf{Lower bound.}
Let $\delta\in(0,\infty)$. Similarly to (\ref{wicm45})-(\ref{7t6ij6y67}), we have
\begin{align}\label{rfervr3454}
&\int_{\mathbb{R}^d}\mathbb{P}_{\delta_x}\left(Z_n(B(b_nr))>0\right)dx\cr
&\leq \mathbb{P}_{\delta_0}(|Z_n|>0)b^d_nv_d(r+\delta)+b^d_n\int_{|y|\geq r+\delta}\mathbb{P}\left(|W_n|\geq b_n(|y|-r)\right)dy.
\end{align}
Recall that $W^{(1)}_n$ is the first component of $W_n$. By Lemma \ref{Nagaevlem}, if $\mathbb{E}[|X|^{\alpha}]<\infty$ for some $\alpha>1$,
\begin{align}
b^d_n\int_{|y|\geq r+\delta}\mathbb{P}\left(|W_n|\geq b_n(|y|-r)\right)dy
&\leq db^d_n\int_{|y|\geq r+\delta}\mathbb{P}\left(|W^{(1)}_n|\geq b_n(|y|-r)\right)dy\cr
&\leq db^d_n\int_{|y|\geq r+\delta}nb^{-\alpha}_n(|y|-r)^{-\alpha}dy\cr
&\leq dnb^{d-\alpha}_n\int_{u\geq \delta}u^{-\alpha}u^{d-1}du.\nonumber
\end{align}
Similarly to (\ref{4tfs3e3})-(\ref{354rde}), one can show that if $\alpha>(\beta+1)d$,
$$
\lim_{n\to\infty}nb^{d-\alpha}_n=0.
$$
Thus, the second term on the right hand side of (\ref{rfervr3454}) converges to $0$ as $n\to\infty.$ This, together with (\ref{rfervr3454}) and (\ref{tgbtsg4}), implies that
$$\liminf_{n\to\infty}\mathbb{P}\left(\frac{R_n}{b_n}\geq r\right)\geq e^{-v_d(r)(\beta^{-1})^{\beta^{-1}}}.$$
\end{proof}

\section{Proof of Theorem \ref{thdim-subcrit}: $m<1$}\label{secsubcrtical}
Again, the proof of this theorem is similar to that for Theorem \ref{thdim1} in spirit.
\begin{proof}
\textbf{Upper bound}.
Similarly to (\ref{jitbndf12}), we have
\begin{align}
\mathbb{P}\left(\frac{R_n}{(1/m)^\frac{n}{d}}\geq r\right)=\exp\left\{-\int_{\mathbb{R}^d}\mathbb{P}_{\delta_x}(Z_n(B(m^{-n/d}r))>0)dx\right\}.
\end{align}
Fix $\delta\in(0,r)$. Similarly to (\ref{54tfrf2}),
\begin{align}\label{4fsw3d}
\int_{\mathbb{R}^d}\mathbb{P}_{\delta_x}(Z_n(B(m^{-n/d}r))>0)dx&\geq\int_{|x|\leq m^{-n/d}(r-\delta) }\mathbb{P}_{\delta_x}(Z_n(B(m^{-n/d}r))>0)dx\cr
&\geq\int_{|x|\leq m^{-n/d}(r-\delta) }\mathbb{P}_{\delta_x}(|Z_n|>0)\mathbb{P}(|W_n+x|\leq m^{-n/d}r )dx\cr
&=\mathbb{P}_{\delta_0}(|Z_n|>0)\int_{|y|\leq r-\delta }\mathbb{P}(|W_n+m^{-n/d}y|\leq m^{-n/d}r )m^{-n}dy\cr
&\geq\mathbb{P}_{\delta_0}(|Z_n|>0)m^{-n}\int_{|y|\leq r-\delta }\mathbb{P}\left(\left|\frac{W_n}{m^{-n/d}}+y\right|\leq r \right)dy\cr
&\geq\mathbb{P}_{\delta_0}(|Z_n|>0)m^{-n}v_d(r-\delta)\mathbb{P}\left(\left|\frac{W_n}{m^{-n/d}}\right|\leq \delta \right).
\end{align}
It is simple to see that
\begin{align}\label{5terswe}
\mathbb{P}\left(\left|\frac{W_n}{m^{-n/d}}\right|\leq \delta \right)&\geq\mathbb{P}\left(\left|\frac{S^{(1)}_n}{m^{-n/d}}\right|\leq \delta \right)^d\cr
&\geq\mathbb{P}\left(|X^{(1)}_{i}|\leq m^{-n/d}\delta/n~,\forall i=1,2,...,n\right)^d\cr
&\geq\left[1-\mathbb{P}\left(|X^{(1)}|\geq m^{-n/d}\delta/n\right)\right]^{nd}.
\frac{}{}\end{align}
If $\mathbb{E}[|X|^{\alpha}]<\infty$ for some $\alpha>0$, then by Markov inequality,
\begin{align}
\mathbb{P}\left(|X^{(1)}|\geq m^{-n/d}\delta/n\right)\leq \left(\frac{n}{m^{-n/d}\delta}\right)^{\alpha}\mathbb{E}\left[|X^{(1)}|^{\alpha}\right].\nonumber
\end{align}
Plugging above into (\ref{5terswe}) yields that
$$
\lim_{n\to\infty}\mathbb{P}\left(\left|\frac{W_n}{m^{-n/d}}\right|\leq \delta \right)=1.
$$
Thus, from (\ref{4fsw3d}) and Lemma \ref{extinpro1}, we obatin

\begin{align}
\limsup_{n\to\infty}\mathbb{P}\left(\frac{R_n}{(1/m)^\frac{n}{d}}\geq r\right)&=\exp\left\{-\int_{\mathbb{R}^d}\mathbb{P}_{\delta_x}(Z_n(B(m^{-n/d}r))>0)dx\right\}\cr
&\leq\exp\left\{-v_d(r-\delta)\lim_{n\to\infty}\mathbb{P}_{\delta_0}(|Z_n|>0)m^{-n}\right\}\cr
&\leq\exp\left\{-v_d(r-\delta)Q(0)\right\},\nonumber
\end{align}
which implies the desired upper bound by letting $\delta\to0$.

\smallskip
\noindent\textbf{Lower bound}.
Fix $\delta>0$, observe that
\begin{align}\label{rvre341}
&\int_{\mathbb{R}^d}\mathbb{P}_{\delta_x}(Z_n(B(m^{-n/d}r))>0)dx\cr
&\leq\int_{|x|\leq m^{-n/d}(r+\delta)}\mathbb{P}_{\delta_x}(|Z_n|>0)dx+\int_{|x|> m^{-n/d}(r+\delta)}\mathbb{P}_{\delta_x}(Z_n(B(m^{-n/d}r))>0)dx\cr
&\leq\mathbb{P}_{\delta_0}(|Z_n|>0)m^{-n}v_d(r+\delta)+\int_{|x|> m^{-n/d}(r+\delta)}\mathbb{E}_{\delta_0}[|Z_n|]\mathbb{P}\left(|W_n+x|\leq m^{-n/d}r\right)dx\cr
&=\mathbb{P}_{\delta_0}(|Z_n|>0)m^{-n}v_d(r+\delta)+\int_{|y|> r+\delta}\mathbb{P}\left(|W_n+m^{-n/d}y|\leq m^{-n/d}r\right)dy.
\end{align}
If $\mathbb{E}[|X|^{\alpha}]<\infty$ for some $\alpha>1$, then by Markov inequality,
\begin{align}
\mathbb{P}\left(|W_n+m^{-n/d}y|\leq m^{-n/d}r\right)&\leq\mathbb{P}\left(|W_n|\geq m^{-n/d}(|y|-r)\right)\cr
&\leq\mathbb{P}\left(\exists 1\leq i\leq n~\text{s.t.}~|X_i|\geq m^{-n/d}(|y|-r)/n\right)\cr
&\leq n\mathbb{P}\left(|X|\geq m^{-n/d}(|y|-r)/n\right)\cr
&\leq\frac{n^{1+\alpha}}{m^{-n\alpha/d}}\frac{\mathbb{E}[|X|^{\alpha}]}{(|y|-r)^{\alpha}},
\end{align}
which implies the second term on the r.h.s. of (\ref{rvre341}) tends to $0$. This, together with (\ref{rvre341}), yields that
\begin{align}
\liminf_{n\to\infty}\mathbb{P}\left(\frac{R_n}{(1/m)^\frac{n}{d}}\geq r\right)&=\exp\left\{-\int_{\mathbb{R}^d}\mathbb{P}_{\delta_x}(Z_n(B(m^{-n/d}r))>0)dx\right\}\cr
&\geq\exp\left\{-v_d(r+\delta)\lim_{n\to\infty}\mathbb{P}_{\delta_0}(|Z_n|>0)m^{-n}\right\}\cr
&=\exp\left\{-v_d(r+\delta)Q(0)\right\},\nonumber
\end{align}
where the last equality follows from Lemma \ref{extinpro1}. The desired lower bound follows by letting $\delta\to0$.
\end{proof}

\textbf{Acknowledgements} The first author's research is supported in part by NSFC grants 61873325,
11831010 and Southern University of Science and Technology Start up found Y01286120.

\smallskip
\textbf{Competing Interests} There were no competing interests to declare which arose during
the preparation or publication process of this article.

\smallskip
\textbf{Availability of data and materials} Data sharing is not applicable to this article as no
datasets were generated or analyzed during the current study.\\

\vspace{0.2cm}
Jie Xiong\\
Department of Mathematics, Southern University of Science and Technology, Shenzhen, China\\
E-mail: xiongj@sustech.edu.cn

\bigskip
\noindent Shuxiong Zhang\\
Department of Mathematics, Southern University of Science and Technology, Shenzhen, China\\
E-mail: shuxiong.zhang@mail.bnu.edu.cn

\begin{thebibliography}{4}
\bibitem[\protect\citeauthoryear{Athreya and Ney}{1972}]{r1}
K. B. Athreya and P. E. Ney. (1972). \textit{Branching Processes}.
Berlin: Springer.

\bibitem[\protect\citeauthoryear{Bovier}{2016}]{Bovier}
 A. Bovier. (2016).
\textit{ Gaussian Processes on Trees: From Spin Glasses to Branching Brownian Motion.}
{Cambridge: Cambridge University Press.}

\bibitem[\protect\citeauthoryear{Etheridge}{2000}]{Etheridge}
A. M. Etheridge. (2000).
\textit{An Introduction to Superprocess.}
{University Lecture Series, 20. American Mathematical Society, Providence, RI.}

\bibitem[\protect\citeauthoryear{Feller}{1971}]{Feller}
W. Feller. (1971).
\textit{An Introduction to Probability Theory and Its Applications II, 2nd edition.}
{New York: John Wiley and Sons}.




\bibitem[\protect\citeauthoryear{Hu}{2005}]{hu05}
Y. Hu. (2005).
 A note on the empty balls left by a critical branching Wiener process.
\textit{ Periodica Mathematica Hungarica}. \textbf{50} 165--174.

\bibitem[\protect\citeauthoryear{Kesten}{1995}]{Kesten1995}
H. Kesten. (1995).
{Branching random walk with a critical branching part.}
\textit{Journal of Theoretical Probability}. \textbf{8} 921--962.

\bibitem[\protect\citeauthoryear{Lalley and Shao}{2015}]{lalley2015}
S. P. Lalley and Y. Shao. (2015).
{On the maximal displacement of critical branching random walk.}
\textit{ Probability Theory and Related Fields}. \textbf{162} 71--96.


\bibitem[\protect\citeauthoryear{Le Gall}{1999}]{LeGall}
J.-F. Le Gall. (1999).
\textit{ Spatial Branching Processes, Random Snakes and Partial Differential Equations}.
{Lectures in Mathematics ETH Z\"urich. Basel: Birkh\"auser.}

\bibitem[\protect\citeauthoryear{Li}{2011}]{Li}
 Z. Li. (2011).
 \textit{ Measure-Valued Branching Markov Processes.}
 {Heidelberg: Springer.}


\bibitem[\protect\citeauthoryear{Nagaev}{1979}]{Nagaev}
S. V. Nagaev. (1979).
 Large deviations of sums of independent random variables.
\textit{ The Annals of Probability}. \textbf{7} 745--789.

\bibitem[\protect\citeauthoryear{Perkins}{2002}]{perkins}
E. A. Perkins. (2002).
\textit{ Dawson-Watanabe Superprocesses and Measure-Valued Diffusions.}
{Berlin: Springer}.


\bibitem[\protect\citeauthoryear{R\'ev\'esz}{2002}]{reves02}
P. R\'ev\'esz. (2002).
 Large balls left empty by a critical branching Wiener field.
\textit{ Statistica Neerlandica}. \textbf{56} 195--205.

\bibitem[\protect\citeauthoryear{Shi}{2015}]{zhan}
Z. Shi. (2015).
\textit{ Branching Random Walks.}
{\'Ecole d'\'Et\'e de Probabilit\'es de Saint-Flour XLII-2012}. Lecture Notes in Mathematics 2151. Berlin: Springer.

\bibitem[\protect\citeauthoryear{Slack}{1968}]{Slack}
R. S. Slack. (1968).
{A branching process with mean one and possibly infinite variance.}
\textit{ Zeitschrift f\"ur Wahrscheinlichkeitstheorie und Verwandte Gebiete}. \textbf{9} 139--145.

\bibitem[\protect\citeauthoryear{Xiong and Zhang}{2022+}]{xz21}
J. Xiong and S. Zhang. (2022+).
 On the empty balls of a critical super-Brownian motion.
\textit{arXiv.}2204.11468.


\end{thebibliography}
\end{document}